\documentclass{amsart}

\usepackage{amsmath,amssymb,amsthm}
\usepackage{graphicx,tikz}
\newtheorem{theorem}{Theorem}
\newtheorem*{thm}{Theorem}
\newtheorem*{proposition}{Proposition}
\newtheorem{corollary}{Corollary}

\newtheorem{lemma}{Lemma}

\theoremstyle{definition}

\theoremstyle{remark}

\begin{document}

\title[]{Positive-definite Functions, Exponential Sums and the Greedy Algorithm: a curious phenomenon}
\keywords{Low discrepancy sequence, Wasserstein distance, Exponential Sum, Erd\H{o}s-Tur\'an inequality, Green function, Irregularities of Distribution, Approximation Theory.}
\subjclass[2010]{11L03, 31C20, 35J08, 41A25, 42B05.}

\author[]{Louis Brown}
\address{Department of Mathematics, Yale University, New Haven, CT 06511, USA}
\email{louis.brown@yale.edu}

\author[]{Stefan Steinerberger}
\address{Department of Mathematics, Yale University, New Haven, CT 06511, USA}
\email{stefan.steinerberger@yale.edu}

\thanks{This paper is part of the first author's PhD thesis, he gratefully acknowledges Travel Support from Yale Graduate School. The second author is supported by the NSF (DMS-1763179) and the Alfred P. Sloan Foundation.}

\begin{abstract} We describe a curious dynamical system that results in sequences of real numbers in $[0,1]$ with seemingly remarkable properties. Let the even function $f:\mathbb{T} \rightarrow \mathbb{R}$ satisfy $\widehat{f}(k) \geq c|k|^{-2}$ and define a sequence via
$$ x_n = \arg\min_x \sum_{k=1}^{n-1}{f(x-x_k)}.$$
Such sequences $(x_n)_{n=1}^{\infty}$ seem to be astonishingly regularly distributed in various ways (satisfying favorable exponential sum estimates; every interval $J \subset [0,1]$ contains $\sim |J|n$ elements). We prove
$$ W_2\left( \mu, \nu\right) \leq \frac{c}{\sqrt{n}}, \quad \mbox{where}~\mu = \frac{1}{n} \sum_{k=1}^{n} \delta_{x_k}$$
is the empirical distribution, $\nu = dx$ is the Lebesgue measure and $W_2(\mu, \nu)$ is the 2-Wasserstein distance between these two.
Much stronger results seem to be true and it is an interesting problem to understand this dynamical system better. We obtain optimal results in dimension $d \geq 3$: using $G(x,y)$ to denote the Green's function of the Laplacian on a compact manifold, we show that 
$$ x_n = \arg\min_{x \in M} \sum_{k=1}^{n-1}{G(x,x_k)} \quad \mbox{satisfies} \quad W_2\left( \frac{1}{n} \sum_{k=1}^{n}{\delta_{x_k}}, dx\right) \lesssim \frac{1}{n^{1/d}}.$$
\end{abstract}

\maketitle

\section{A Curious Phenomenon}
\subsection{Introduction.} The purpose of this paper is to discuss a phenomenon that is connected to problems in Harmonic Analysis, Combinatorics, Number Theory and Approximation Theory. Perhaps the simplest instance of the phenomenon that we do not understand is the following: let $f:\mathbb{T} \rightarrow \mathbb{R}$ be an even function with mean value 0 and satisfy $\widehat{f}(k) \geq c |k|^{-2}$ for some fixed constant $c>0$ and all $k \neq 0$ (this is a quantitative form of the function being positive definite). A simple example of such a function is the second Bernoulli polynomial which, identifying $\mathbb{T} \cong [0,1]$, can be written as
$$ f(x) = x^2 - x + \frac{1}{6}.$$

 We now define a sequence of points by starting with an arbitrary initial set of points $\left\{x_1, \dots, x_m \right\} \subset [0,1]$ and then setting, in a greedy fashion,
$$ x_n = \arg\min_{x \in \mathbb{T}} \sum_{k=1}^{n-1}{f(x-x_k)}.$$
Here and throughout the paper we adopt the convention that if the minimum is not attained at a unique point, then any of the points can be chosen. The main contribution of our paper is to point out that such sequences seem
to enjoy \textit{remarkable} distribution properties; these observations are mainly empirical at this stage. We refer to the papers \cite{stein, stein2} for some numerical experiments (see also below). 
We summarize the existing results and derive some new ones; however, the overall phenomenon is largely unexplained.
\begin{enumerate}
\item \textbf{Open Problem 1.} Is it true that
$$ \sum_{k, \ell = 1}^{n}{ f(x_k - x_{\ell})} \lesssim \log{n}?$$
\item \textbf{Open Problem 2.} Is it true that
$$ \left\| \sum_{k=1}^{n}{f(x - x_k)} \right\|_{L^{\infty}} \lesssim \log{n}?$$
\end{enumerate}
Before embarking on a discussion of these problems, we quickly illustrate Open Problem 2 with a simple example. Indeed, Open Problem 2 can be stated in very simple terms. We fix again $ f(x) = x^2 - x + 1/6$ and obtain
a sequence by starting with $x_1 = 0.3, x_2 = 0.8$ and using the greedy algorithm to obtain all subsequent elements of the sequence. We also abbreviate
$$ f_n(x) = \sum_{k=1}^{n}{f(x-x_k)}.$$
\vspace{-5pt}
\begin{figure}[h!]
\begin{tikzpicture}[scale=1.2]
\node at(0,0) {\includegraphics[width = 8cm]{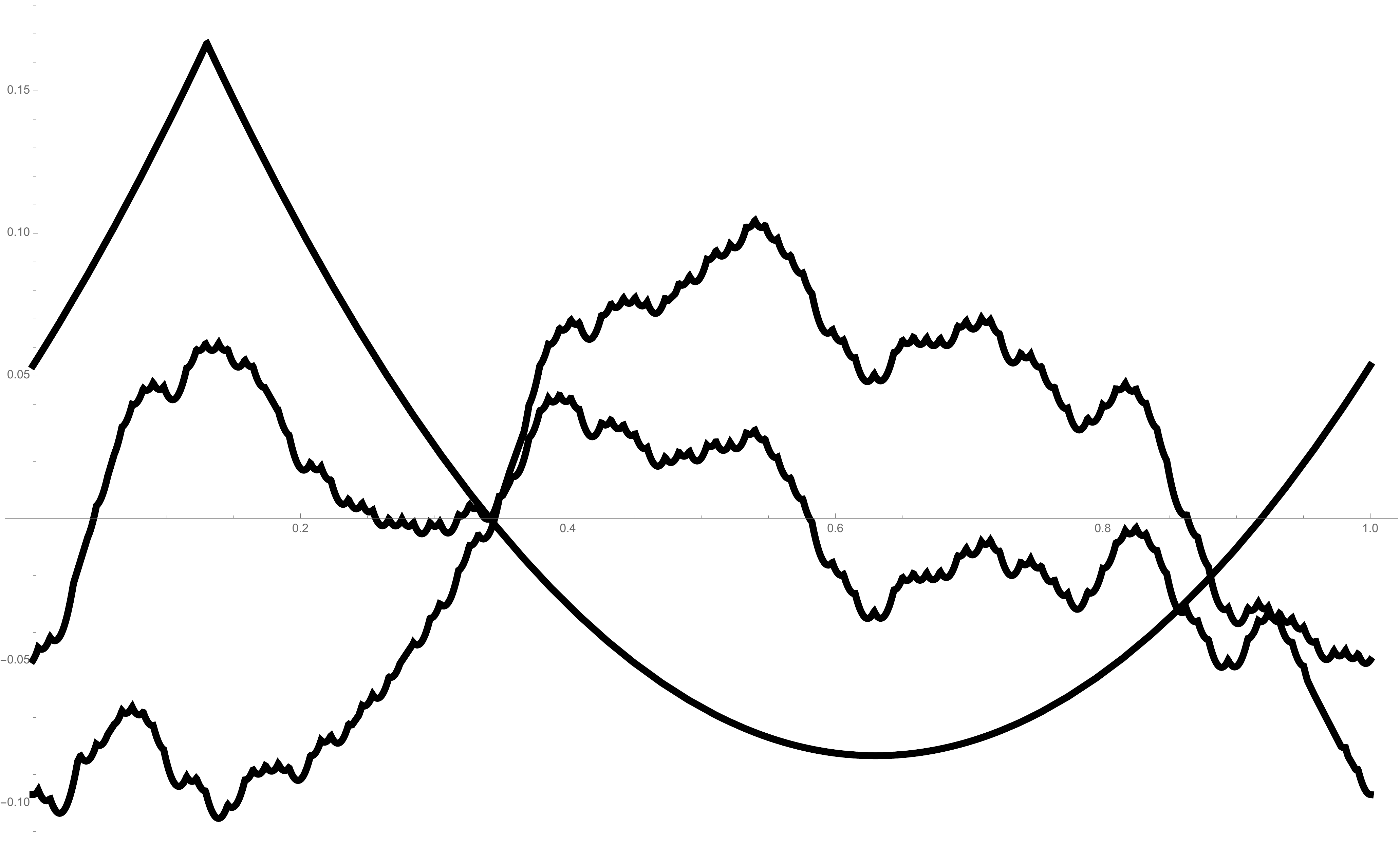}};
\node at (-1,-2.4) {$f_{101}$ has its minimum here};
\node at (-0.5,-1.5) {$f_{101}$};
\draw [thick, ->] (-0.8, -1.5) -- (-1.2,-1.4);
\draw[thick] (-2.3,-1.7) circle (0.4cm);
\node at (-3.8, -0.2) {$f_{102}$};
\draw [thick, ->] (-3.5,-0.2) -- (-2.9, -0.2);
\node at (-0.8, 1.5) {$f(x-x_{101})$};
\end{tikzpicture}
\caption{We obtain $f_{102}$ by finding the point $x_{101}$ at which $f_{101}$ assumes its minimum and $f_{102}(x) = f_{101}(x) + f(x-x_{101})$. }
\end{figure}
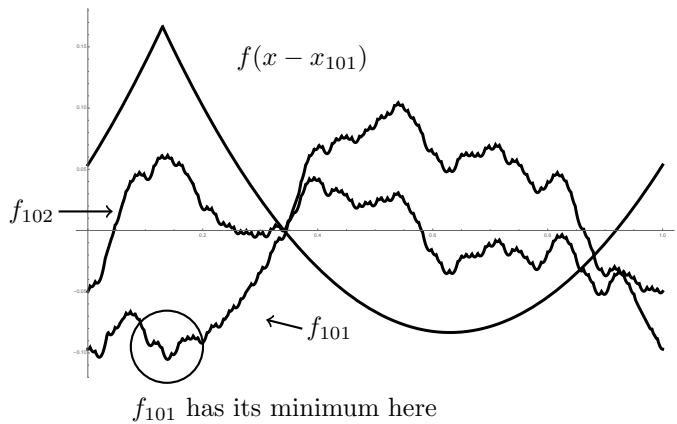

As seen in Figure 1, the function $f_n$ does not seem to be very large: this is only possible if the sequence elements are so regular that the sum over $f(x-x_k)$ leads to good cancellation properties. This is one instance of the `curious' phenomenon alluded to in the title: why is $\|f_n\|_{L^{\infty}}$ so remarkably small in $n$?

\begin{figure}[h!]
\begin{tikzpicture}[scale=1.2]
\node at(0,0) {\includegraphics[width = 8cm]{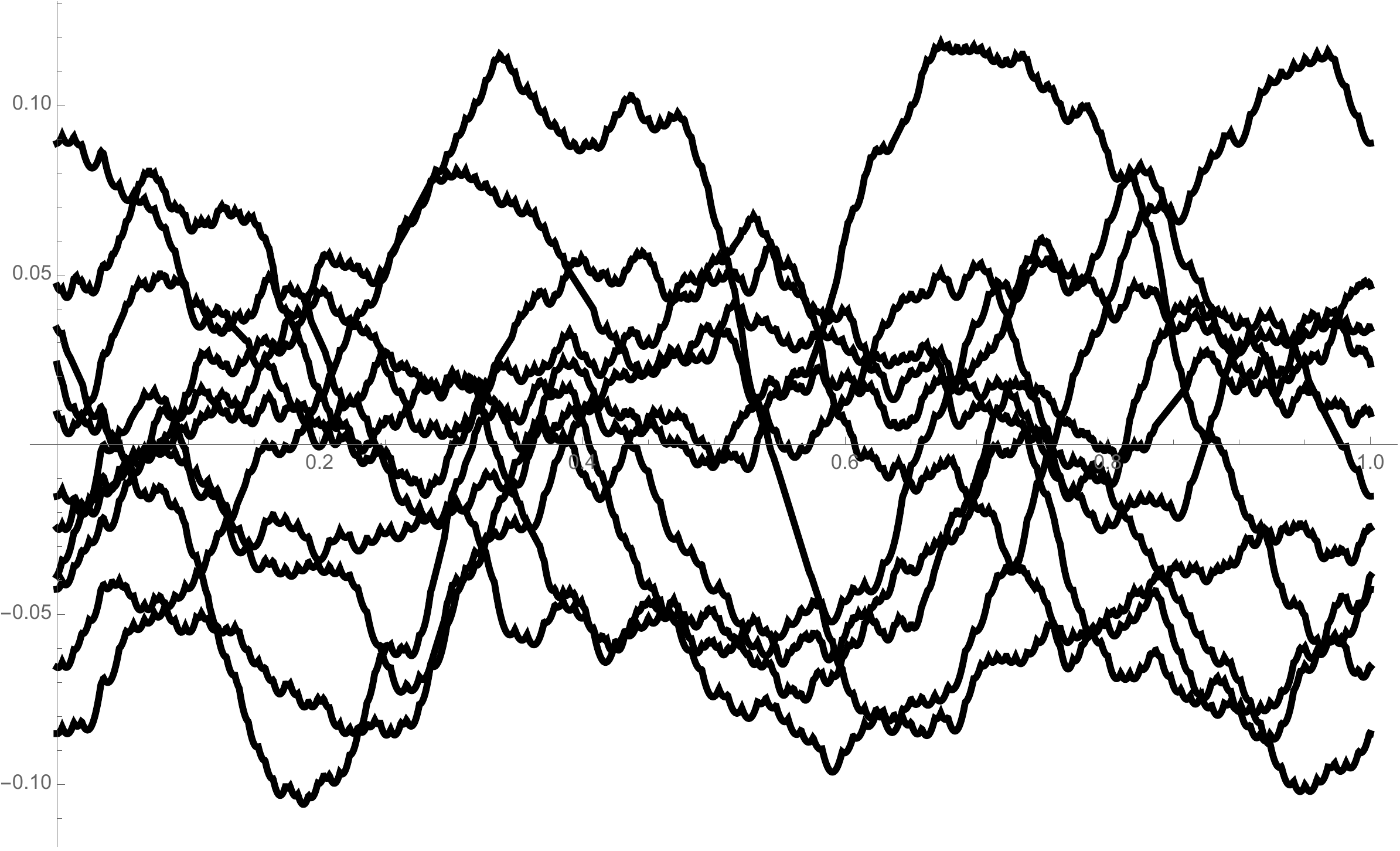}};
\end{tikzpicture}
\caption{The functions $f_{100}, f_{110}, f_{120} \dots, f_{200}$. We observe that they are quite different from one another and have an interesting behavior. Most importantly, they all seem to be quite small with $\|f_n\|_{L^{\infty}}$ barely exceeding $\|f\|_{L^{\infty}}$.}
\end{figure}

The inequalities posed in Open Problems 1 and 2 above, if true, would indicate that the sequence $(x_n)_{n=1}^{\infty}$ satisfies very good distribution properties. What is remarkable is that, in a certain sense, these properties would be close to optimal and connect to a variety of problems related to several different fields. We will derive below that
$$ \sum_{k, \ell = 1}^{n}{ f(x_k - x_{\ell})} \lesssim n \qquad \mbox{and} \qquad \left\| \sum_{k=1}^{n}{f(x - x_k)} \right\|_{L^{\infty}} \lesssim n^{1/2}$$
but these estimates seem to be very far from sharp. Any improvement of these estimates would immediately yield improvements of our other results via the arguments outlined below.
In the converse direction, it is an interesting problem whether the following is true: if $f:\mathbb{T} \rightarrow \mathbb{R}$ is an even, continuous function with mean 0 such that $\widehat{f}(k) > c|k|^{-2}$ for all $k \neq 0$, is it true that for any sequence $(x_k)_{k=1}^{\infty}$
$$ \left\| \sum_{k=1}^{n}{f(x - x_k)} \right\|_{L^{\infty}} \qquad \mbox{is unbounded in}~n?$$
For the `double' sum, this is a known result of Proinov \cite{proinov} who showed that, for some constant $c_f > 0$ depending only on the function and any sequence $(x_n)_{n=1}^{\infty}$,
$$ \sum_{k, \ell = 1}^{n}{ f(x_k - x_{\ell})}  \geq c_f \log{n} \qquad \mbox{for infinitely many}~n.$$
In particular, this shows that the bounds conjectured in Open Problem 1 would be optimal. 
It seems reasonable to assume that the condition $\widehat{f}(k) \geq c |k|^{-2}$, or some condition like it, is necessary for this phenomenon to occur; it is certainly necessary for our proof that the sequence $(x_n)_{n=1}^{\infty}$ is uniformly distributed. It may be of interest to study the dynamical system when $f$ is a trigonometric polynomial: it seems that in this case the sequence $(x_n)_{n=1}^{\infty}$ will not even be uniformly distributed.

\subsection{Connections To Other Problems.}

We start with a simple example. Let us define, as above, $f(x) = x^2 - x + 1/6$ and consider the sequence obtained via
$$ x_n = \arg\min_{x \in \mathbb{T}} \sum_{k=1}^{n-1}{f(x-x_k)}$$
when starting with $\left\{1/3, 4/5\right\}$. The sequence is easy to compute and starts 
$$ \frac{1}{3}, \frac{4}{5}, 0.066, 0.566, 0.941, 0.441, 0.191, 0.691, \dots$$
Empirically, this sequence (and seemingly any sequence obtained in this way) seems to have remarkable regularity properties. These properties can be stated in a variety of different ways (all of which are mutually connected), we refer to \cite{bc, dick, drmota, kuipers}.
\begin{itemize}
\item \textbf{Combinatorial.} For every $n \in \mathbb{N}$, the set $\left\{x_1, \dots, x_n\right\}$ has the property that for every interval $J \subset [0,1]$, the number of elements in $J$ is $|J| \cdot n$ with a very small error. It is known \cite{sch} that the error has to be at least of size $\gtrsim \log{n}$ for infinitely many values of $n$ and there are constructions where the error is indeed uniformly bounded by $\sim \log{n}$ in $n$.\\
\item \textbf{Analytical} (Erd\H{o}s-Tur\'an \cite{erd1, erd2})\textbf{.} The sequence has the property that $\left\{x_1, \dots, x_n\right\}$ satisfy favorable exponential sum estimates on expressions of the form
$$ \sum_{k=1}^{n}{ \frac{1}{k} \left| \sum_{\ell=1}^{n}{ e^{2\pi i k x_{\ell}}} \right|} \qquad \mbox{and} \qquad  \sum_{k=1}^{n}{ \frac{1}{k^2} \left| \sum_{\ell=1}^{n}{ e^{2\pi i k x_{\ell}}} \right|^2}.$$
The exponential sum $\sum_{\ell=1}^{n}{\exp(2 \pi i k x_{\ell})}$ is `small' for `small' values of $k$.\\
\item \textbf{Numerical} (Koksma-Hlawka \cite{hlawka})\textbf{.} The set $\left\{x_1, \dots, x_n\right\}$ is a good set for numerical integration: we have
$$ \int_{0}^{1}{f(x)dx} \sim  \frac{1}{n}\sum_{k=1}^{n}{f(x_k)}$$
with a `small' error for `smooth' functions $f$.\\
\item \textbf{Geometric} (Roth \cite{roth})\textbf{.} The two-dimensional set 
$$ \left\{ \left( \frac{i}{n}, x_i\right): 1 \leq i \leq n \right\} \subset [0,1]^2$$
is regularly distributed in the unit square: every cartesian box $[a,b] \times [c,d]$ contains roughly $(b-a)(d-c)n$ elements with a small error (see Fig. 3).
\end{itemize}

These problems have been intensively studied for over a century starting with the seminal paper of Weyl \cite{weyl}. We refer to the foundational results \cite{aard, beck, bil3, bil4, erd1,erd2, hlawka, roth, sch}, the survey paper \cite{bil} and the textbooks \cite{bc, chazelle, dick, drmota, kuipers} (also with regard to various different ways of interpreting the notion of `small' and `smooth' in the above statements and to which extent they are connected to one another).\\

\begin{figure}[h!]
\begin{tikzpicture}[scale=1.2]
\node at(0,0) {\includegraphics[width = 6.6cm]{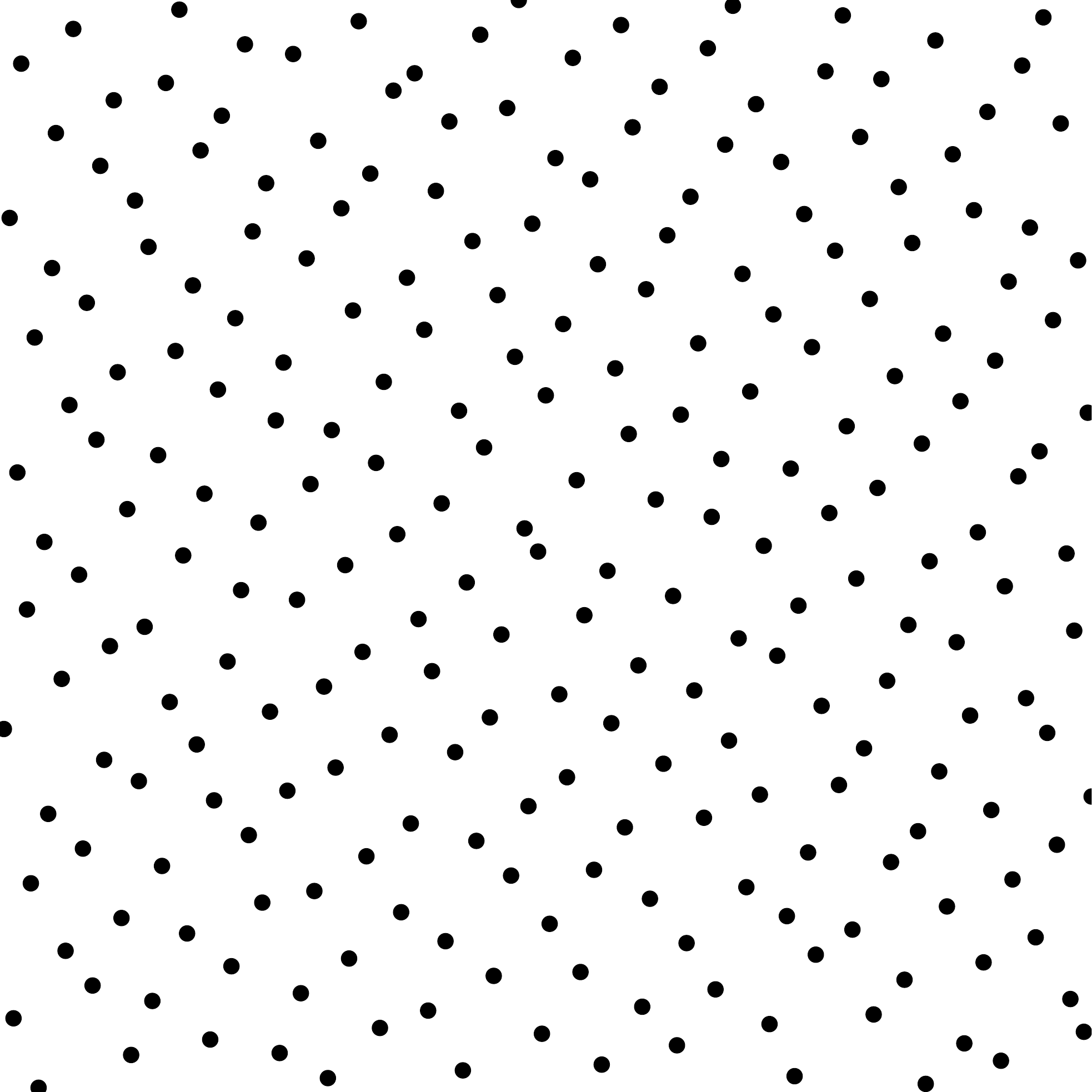}};
\draw [thick] (-2.78,-2.78) -- (2.78,-2.78) -- (2.78,2.78) -- (-2.78,2.78) -- (-2.78,-2.78);
\end{tikzpicture}
\caption{250 points created starting with $\left\{1/3, 4/5\right\}$ and using the second Bernoulli polynomial for $f$. We display the points $(n/250, x_n) \in [0,1]^2$ for $1\leq n \leq 250$. Why is this distribution so regular?}
\end{figure}

We introduce two very structurally different examples of sequences that are both known to be optimally behaved with respect to all of the above properties. These two competing examples are the
\begin{enumerate}
\item \textit{Kronecker sequence} given by $x_n = \left\{n \sqrt{2}\right\}$ (where $\left\{ \cdot \right\}$ denotes the fractional part). Here, $\sqrt{2}$ could be replaced by any other number with bounded continued fraction expansion.
\item \textit{van der Corput sequence} given by taking $x_n$ to be the rational number whose binary expansion is the reversed string of bits of $n$ when written in binary. It is known that base 2 can be replaced by any other prime basis.
\end{enumerate}

 We can now state our question as follows. 
\begin{quote}
\textbf{Open Problem 3.} Are these greedy sequences, up to constants, comparable to the behavior of the best Kronecker sequence or the van der Corput sequence in all the ways outlined above?
\end{quote}
If this were indeed the case, it could have very interesting consequences. Both the Kronecker sequence and the van der Corput sequence are known to be optimal in the one-dimensional setting (a result of Schmidt \cite{sch}, see also \cite{larcher0, larcher} for an improved constant). However, nobody knows what sequences are optimal in even $d=2$ dimensions (we refer to the excellent survey of Bilyk \cite{bil}). So, if there was a greedy-type construction with optimal behavior in $d=1$ dimension, it might suggest sequences of similar quality in higher dimensions as well -- this would be interesting because the greedy sequence seems to be unlike any that has been studied; in particular, if it enjoys good distribution properties, this seems like it would have to be because of a different underlying mechanism.

\subsection{Known results.} This type of construction was first proposed by the second author in \cite{stein}. There it was shown that if the function is 
$$ f(x) = -\log{\left(2 \sin{(\pi |x|)} \right)},$$
then the arising sequence satisfies, for all intervals $J \subset [0,1]$,
$$ \left| \frac{\#\left\{1 \leq i \leq n: x_i \in J \right\} }{n} - |J| \right| \leq  c\frac{\log{n}}{\sqrt{n}}.$$
 The arguments are based on the explicit structure of the Fourier series of $f(x)$ and do not generalize to other functions. It is already discussed in \cite{stein} that much stronger results seem to be true and that the sequence arising from this function $f$ seems, numerically, as well behaved with regard to all these aspects as the Kronecker or the van der Corput sequence. The same idea, interpreted differently, has also led to a numerical scheme that seems to be effective at regularizing point sets \cite{stein2}. 
 It was also noted in \cite{stein} that if one starts with a single element $\left\{x_1\right\}$, then the arising sequence seems to be related to the van der Corput sequence -- this is indeed the case and was subsequently proven by Pausinger \cite{pausinger}. Pausinger's theorem holds for the much larger family of strictly convex functions $f:[0,1]\rightarrow \mathbb{R}$ that are symmetric around $x=1/2$. There it is also shown that for general functions of this type, the arising construction results in a sequence satisfying for all $J \subset [0,1]$,
 $$ \left| \frac{\#\left\{1 \leq i \leq n: x_i \in J \right\} }{n} - |J| \right| \leq  \frac{c}{n^{1/3}}.$$
The bound stated in Open Problem 1 would improve this estimate to $\lesssim \log{n}/n^{2/3}$ (which, however, is still not at the $\log{n}/n$ level that we observe numerically).
Other types of greedy constructions of sequences have been considered in the literature, we refer to work of Kakutani \cite{kaku} and Temylakov: \cite{temr1, temr2} and \S 6.11 in \cite{tem1}. [Note added in print: Temlyakov has since used this type of sequence to establish an endpoint result for a result in Numerical Integration \cite{temnew}.]

\section{Results}
\subsection{Wasserstein distance.} The main purpose of this paper is to (1) describe the phenomenon and its connections in a concise way and (2) to point out that we can obtain slightly improved regularity results by switching to the Wasserstein distance. The Wasserstein distance \cite{wasser, villani} is a notion of distance between two measures (roughly `how much mass has to be transported how far from an initial measure to achieve a target measure'). We will only discuss the case where one measure is the empirical distribution 
$$ \mu = \frac{1}{n} \sum_{k=1}^{n}{ \delta_{x_k}} \qquad \mbox{and the other measure is} \qquad \nu = dx.$$ 
The $p-$Wasserstein distance between two measures
$\mu$ and $\nu$ is defined as
$$ W_p(\mu, \nu) = \left( \inf_{\gamma \in \Gamma(\mu, \nu)} \int_{M \times M}{ |x-y|^p d \gamma(x,y)}\right)^{1/p},$$
where $| \cdot |$ is the metric and $\Gamma(\mu, \nu)$ denotes the collection of all measures on $M \times M$
with marginals $\mu$ and $\nu$, respectively (also called the set of all couplings of $\mu$ and $\nu$). 
In our setting, we trivially have $0 \leq W_p(\mu, \nu) \leq 1$. 
The known inequality \cite{pausinger}
 $$ \left| \frac{\#\left\{1 \leq i \leq n: x_i \in J \right\} }{n} - |J| \right| \leq  \frac{c}{n^{1/3}}$$
can be coupled with the Monge-Kantorovich formula (see e.g. \cite{villani}) to obtain
$$ W_1\left( \frac{1}{n} \sum_{k=1}^{n} \delta_{x_k} , dx \right) \lesssim \frac{1}{n^{1/3}}.$$

\subsection{Main Results.}
Our main result is an improvement for the $W_2-$distance. H\"older's inequality shows that
$W_1(\mu, \nu) \leq W_2(\mu, \nu)$, so the result also implies improved bounds for the $W_1$ distance.

\begin{theorem}  Let the even function $f:\mathbb{T} \rightarrow \mathbb{R}$ satisfy $\widehat{f}(k) \geq c |k|^{-2}$ for some fixed constant $c>0$ and all $k \neq 0$. Define a sequence via
$$ x_n = \arg\min_x \sum_{k=1}^{n-1}{f(x-x_k)},$$
then this sequence satisfies
$$ W_2\left( \frac{1}{n} \sum_{k=1}^{n} \delta_{x_k} , dx \right) \lesssim \frac{1}{n^{1/2}},$$
where the implicit constant depends only on the initial set, $f(0)$ and $c$.
\end{theorem}

One way of interpreting the Theorem is as follows: given $\left\{x_1, \dots, x_n\right\}$ we can interpret these points as Dirac measures with weight $1/n$. It is then possible to `break' these points up and move their $L^1$-mass a distance of, on average, not more than $\sim n^{-1/2}$ to recreate the uniform distribution. The result seems to be far from the truth, which we believe to be at scale $n^{-1}$ up to logarithmic factors (see below).
We also obtain the following corollary (which was suggested to us together with its proof by Igor Shparlinski).
\begin{corollary} Suppose $f:\mathbb{T} \rightarrow \mathbb{R}$ is even, has mean value 0 and satisfies both
$$\widehat{f}(k) > 0 \qquad \mbox{and} \qquad \sum_{k \in \mathbb{Z}} \widehat{f}(k) < \infty.$$

Then, for any sequence $(x_n)$ arising from the algorithm outlined above,
$$\left\| \sum_{k=1}^{n}{f(x - x_k)} \right\|_{L^{\infty}} \lesssim \sqrt{n}.$$
\end{corollary}
Again, we believe this to be far from optimal and expect the quantity to grow not much faster than (at most) logarithmically.
We have a slight refinement of this statement in the case $\widehat{f}(k) \sim |k|^{-2}$ 

\begin{theorem} Let $f:\mathbb{T} \rightarrow \mathbb{R}$ be an even function with mean 0 satisfying $c_1 |k|^{-2} \leq \widehat{f}(k) \leq c_2 |k|^{-2}$ for all $k \neq 0$ for some universal $c_1, c_2 > 0$. Then
$$\left\| \sum_{k=1}^{n}{f(x - x_k)} \right\|_{L^{\infty}} \lesssim n^{1/3} \qquad \mbox{for infinitely many}~n.$$
\end{theorem}
The argument is slightly finer than this: we will prove that 
$$\left\| \sum_{k=1}^{n}{f(x - x_k)} \right\|_{L^{\infty}} \lesssim n^{1/3} \left\| \sum_{k=1}^{n}{f(x - x_k)} \right\|^{1/3}_{L^{1}}$$
and then prove that the $L^1-$term has to be $\lesssim 1$ infinitely many times.
We note that this result is below the $n^{1/2}-$threshold that we would expect from randomly chosen points. Again, as mentioned above, we expect the error rate to actually be much smaller than this.\\

We will now discuss why Wasserstein distance is a very canonical way of capturing problems of this type. 
We state this formally in the following estimate.

\begin{corollary} Suppose $f:\mathbb{T} \rightarrow \mathbb{R}$ is even, has mean value 0 and satisfies $\widehat{f}(k)\geq c |k|^{-2}$ for $k \neq 0$. Then, for any set $\left\{x_1, \dots, x_n\right\} \subset \mathbb{T}$, we have
$$ W_2\left( \frac{1}{n} \sum_{k=1}^{n} \delta_{x_k} , dx \right) \lesssim \frac{1}{n} \left(\sum_{k,\ell=1}^{n}{f(x_k - x_{\ell})}\right)^{1/2}.$$
\end{corollary}

Fix now a function such that $\widehat{f}(k) \sim |k|^{-2}$ for $k\neq 0$ (in the sense of having corresponding upper and lower bounds). Open Problem 1 asks whether
$$ \sum_{k, \ell = 1}^{n}{ f(x_k - x_{\ell})} \lesssim \log{n} \qquad \mbox{might hold}$$
and, conversely, which kind of lower bounds exist. Corollary 2 shows that any such estimate would imply
$$ W_2\left( \frac{1}{n} \sum_{k=1}^{n} \delta_{x_k} , dx \right) \lesssim \frac{\sqrt{\log n}}{n}.$$
This connects to yet another problem, that of \textit{irregularities of distribution}. A seminal result of Schmidt \cite{sch} implies that for \textit{any} sequence $(x_n)_{n=1}^{\infty}$ there are infinitely many integers $n$ and intervals $J_n$ such that
$$ \left| \frac{\# \left\{x_1, \dots, x_n:x_i \in J_n\right\}}{n} - |J_n|\right| \geq \frac{1}{100} \frac{\log{n}}{n}.$$
This result shows that, in a sense, irregularities of distribution are unavoidable. 
A natural question now is the following: does a similar phenomenon exist for the Wasserstein distance?
This was answered by Cole Graham \cite{graham} who proved the following result.

\begin{thm}[Graham \cite{graham}] For any sequence $(x_n)_{n=1}^{\infty}$ in $[0,1]$, we have
$$ W_2\left( \frac{1}{n} \sum_{k=1}^{n} \delta_{x_k} , dx \right) \gtrsim \frac{ \sqrt{ \log{n}}}{n}\quad \mbox{for infinitely many}~n.$$
\end{thm}
 
The second author has already remarked in \cite{stein0} that this is sharp for the Kronecker sequence $x_n = \left\{n \alpha\right\}$ for any badly approximable $\alpha$ (say, $\alpha = \sqrt{2}$). An implication of Graham's result coupled with our Corollary above is the following result that was first established by Proinov.

\begin{thm}[Proinov, \cite{proinov}] Let $f:\mathbb{T} \rightarrow \mathbb{R}$ be a function with mean value 0 satisfying $\widehat{f}(k) \geq c |k|^{-2}$ for $k \neq 0$. Then, for any sequence $(x_n)_{n=1}^{\infty}$,
we have
 $$ \sum_{k,\ell=1}^{n}{f(x_k - x_{\ell})} \gtrsim \log{n} \qquad \mbox{for infinitely many}~n.$$
\end{thm}

Using again the Kronecker sequence, we can show that there are sequences for which this notion of energy does indeed grow very slowly; this result is folklore, we include it for the convenience of the reader. The same result is also known for the van der Corput sequence, we refer to Proinov \& Grozdanov \cite{proinov2}.
 
 \begin{proposition} Let $f:\mathbb{T} \rightarrow \mathbb{R}$ have mean value 0 satisfying $\widehat{f}(k) \leq c |k|^{-2}$ for $k \neq 0$. Then, for any badly approximable $\alpha$, the sequence $x_n = \left\{n \alpha \right\}$ satisfies
 $$ \sum_{k,\ell=1}^{n}{f(x_k - x_{\ell})} \lesssim \log{n}.$$
 \end{proposition}

The proof of the proposition makes explicit use of a rather delicate property of the sequence $\left\{n \alpha\right\}$. It is thus even more striking that, possibly, the greedy sequence
$$ x_n = \arg\min_x \sum_{k=1}^{n-1}{f(x-x_k)}$$
might conceivably behave in a similar manner. Naturally, this falls into the realm of Approximation Theory and, more specifically, the Greedy Algorithm \cite{dev, tem0, tem1} and its use in Approximation Theory.
Indeed, we can interpret this greedy sequence as a way to approximate the constant function 0 by means of translates $f(x- x_k)$. The Greedy Algorithm is well understood to yield reasonable
estimates for a broad class of functions -- what is of special interest here is that in our case the greedy algorithm seems to perform much better than one would usually expect from a greedy algorithm; moreover, it seems to be comparable in efficiency to subtle constructions in Number Theory that make use of delicate notions such as badly approximable numbers.

\subsection{Two Remarks.} All our estimates are based on the inequality
$$ \sum_{k, \ell = 1}^{n}{f(x_k - x_{\ell})} \leq n f(0).$$
It is not difficult to see (see below) that this is indeed satisfied for our greedy construction. However, the inequality (and therefore our main Theorem) is also valid if $x_n$ is chosen in such a way that
$$ \sum_{k=1}^{n-1}{f(x_n - x_k)} \leq 0.$$
We observe that $f$ has mean value 0 and thus
$$ \int_{\mathbb{T}} \sum_{k=1}^{n-1}{f(x - x_k)}dx = 0$$
and it is always possible to choose a new element $x_n$ with this property (and, usually, there are many of those). However, presumably these elements can be chosen in rather terrible ways and there is no reason to expect these sequences $(x_n)_{n=1}^{\infty}$ to have particularly good distribution properties; it would seem our Theorem is close to optimal for these types of sequences though we do not know how to show this. It also shows the bottleneck in our current approach: we do not know how to make use of the fact that the algorithm chooses the minimal value and not merely a value not exceeding the expected value. The second remark concerns uniform distribution of the sequence $(x_n)_{n=1}^{\infty}$. We have the following fact.
\begin{corollary} If $\widehat{f}(k) > 0$ for all $k \neq 0$, then the sequence $x_n$ defined via
$$ x_n = \arg\min_x \sum_{k=1}^{n-1}{f(x-x_k)}.$$
is uniformly distributed on $\mathbb{T}$.
\end{corollary}
The argument is so short that we can give it right here.
\begin{proof} We have
\begin{align*}
 nf(0) \geq \sum_{m, \ell = 1}^{n}{f(x_m - x_{\ell})} = \sum_{k \in \mathbb{Z} \atop k \neq 0}{\widehat{f}(k) \left|\sum_{m = 1}^{n}e^{2\pi i k x_m}\right|^2  } \geq \widehat{f}(k) \left|\sum_{m = 1}^{n}e^{2\pi i k x_m}\right|^2 
\end{align*}
from which we obtain
$$ \frac{1}{n} \left|\sum_{m = 1}^{n}e^{2\pi i k x_m}\right| \leq  \sqrt{ \frac{f(0)}{\widehat{f}(k)} } \frac{1}{\sqrt{n}}.$$
This tends to 0 from which we obtain uniform distribution from Weyl's theorem.
\end{proof}
We emphasize that the argument also shows that the size of $\widehat{f}(k)$ will play a role in the quality of the distribution: if it decays very rapidly, the convergence rate might be quite slow.

\subsection{Higher dimensions.} The same phenomenon exists in higher dimensions and it does so at a great level of generality.  Indeed, the scaling in higher dimensions is fundamentally different and this allows us to obtain optimal results. Let $(M,g)$ be a smooth compact manifold without boundary. We use $\phi_k$ to denote the $L^2-$normalized eigenfunctions of the Laplace operator
$$ -\Delta \phi_k = \lambda_k \phi_k.$$
We will now define admissible kernels $K:M \times M \rightarrow \mathbb{R}$ to be functions of the form
$$ K(x,y) = \sum_{k=1}^{\infty}{ a_k \frac{ \phi_k(x) \phi_k(y)}{\lambda_k}}$$
where the coefficient $a_k$ is assumed to satisfy a two-sided bound:
$$ c_1 < a_k < c_2 \qquad \mbox{for all}~k \geq 1$$
and some positive constants $c_1, c_2$.
 We note that the sum starts at $k=1$ and thus excludes the trivial (constant) eigenfunction $\phi_0$. In particular, all these kernels have mean value 0. This definition is an extension of our assumption $\widehat{f}(k) \geq c |k|^{-2}$ in the one-dimensional setting.
A particularly natural kernel arises from setting $a_k = 1$ in which case we obtain the Green's function of the Laplacian $G(x,y)$. This function has the property that
$$ -\Delta_x \int_{M}{G(x,y) f(y) dy} = f(x),$$
i.e. it solves the equation $-\Delta u = f$. We will now consider sequences of the form
$$ x_n = \arg\min_{x \in M} \sum_{k=1}^{n-1}{K(x,x_k)}.$$

\begin{theorem} Let $x_n$ be a sequence obtained in such a way on a $d-$dimensional compact manifold. Then 
$$ W_2\left( \frac{1}{n} \sum_{k=1}^{n} \delta_{x_k} , dx \right) \lesssim_M \begin{cases}n^{-1/2} \sqrt{\log{n}} \qquad &\mbox{if}~d=2 \\
n^{-1/d} &\mbox{if}~d \geq 3. \end{cases}$$
\end{theorem}
We note that this result is optimal for $d \geq 3$. We do not know whether the logarithmic factor is necessary for $d=2$. The main ingredient is a favorable estimate of the Wasserstein distance that was recently obtained by the second author \cite{stein3} that allows for a greedy formulation. We note that while the static case, the structure of point
sets minimizing the Green energy, has been an active field of study \cite{beltran, carlos, bet, ch, criado, garcia, lev, marzo, stein3}, we are not aware of results in the dynamic setting.

\begin{corollary} If $d \geq 3$, then there exists a sequence $(x_n)_{n=1}^{\infty}$ on $\mathbb{T}^d$ such that
$$ W_2\left( \frac{1}{n} \sum_{k=1}^{n} \delta_{x_k} , dx \right) \lesssim_d \frac{1}{n^{1/d}} \qquad \mbox{uniformly in}~n.$$
\end{corollary}

This Corollary seems to be new: it gives a constructive proof that Wasserstein distance does not have an irregularities of distribution phenomenon in dimensions $d \geq 3$. We have the same result up to a factor of $\sqrt{\log{n}}$ in two dimensions. By Graham's result \cite{graham}, the loss of a factor of $\sqrt{\log{n}}$ is indeed necessary in $d=1$.

\section{Proofs}
\subsection{Proof of Theorem 1, Corollary 1 and Corollary 2.}
\begin{proof}
The proof decomposes into two parts. In the first part we argue exactly as in \cite{pausinger}. We can assume w.l.o.g. that $f$ has mean value 0.  We first observe that
$$ \sum_{m, \ell = 1}^{n}{f(x_m - x_{\ell})} \leq n f(0).  \qquad \qquad (\diamond)$$
which follows from the identity
$$\sum_{m, \ell = 1}^{n}{f(x_m - x_{\ell})} = n f(0) + 2 \sum_{m, \ell = 1\atop m < \ell}^{n}{f(x_m - x_{\ell})},$$
the reformulation
$$ \sum_{m, \ell = 1\atop m < \ell}^{n}{f(x_m - x_{\ell})} = \sum_{\ell=2}^{n}{ \sum_{m=1}^{\ell-1} f(x_{\ell} - x_m)}$$
and the greedy algorithm: by definition of $x_{\ell}$, we have
$$  \sum_{m=1}^{\ell-1} f(x_{\ell} - x_m) = \min_x  \sum_{m=1}^{\ell-1} f(x - x_m) \leq \int_{\mathbb{T}}  \sum_{m=1}^{\ell-1} f(x_{} - x_m) dx = 0.$$

 Rewriting quantities in terms of Fourier Analysis then shows that
\begin{align*}
 \sum_{m, \ell= 1}^{n}{f(x_m - x_{\ell}) }&=  \sum_{k \in \mathbb{Z}}{\widehat{f}(k) \sum_{m, \ell = 1}^{n}e^{2\pi i k(x_m - x_{\ell})}} \\
 &= \sum_{k \in \mathbb{Z}}{\widehat{f}(k) \left(\sum_{m = 1}^{n}e^{2\pi i k x_m}\right) \left(\sum_{m = 1}^{n}e^{2\pi i k(-x_m)}\right)  }\\
 &=  \sum_{k \in \mathbb{Z}}{\widehat{f}(k) \left(\sum_{m = 1}^{n}e^{2\pi i k x_m}\right) \overline{\left(\sum_{m = 1}^{n}e^{2\pi i kx_m}\right) } }\\
 &=  \sum_{k \in \mathbb{Z}}{\widehat{f}(k) \left|\sum_{m = 1}^{n}e^{2\pi i k x_m}\right|^2 }.
 \end{align*}

We first use this fact to establish the statement of the Corollary 1. This corollary was suggested to us by Igor Shparlinski, and we are grateful to be able to incorporate it here. 
Note that
\begin{align*}
\sum_{\ell=1}^{n}{f(x-x_{\ell})} = \sum_{k \in \mathbb{Z}} \widehat{f}(k) \left( \sum_{\ell=1}^{n}{e^{-2 \pi i k x_{\ell}}} \right) e^{2 \pi i k x}
\end{align*}
and thus, using the Cauchy-Schwarz inequality,
\begin{align*}
\left\| \sum_{\ell=1}^{n}{f(x-x_{\ell})} \right\|_{L^{\infty}} &\leq  \sum_{k \in \mathbb{Z}} \widehat{f}(k) \left| \sum_{\ell=1}^{n}{e^{-2 \pi i k x_{\ell}}} \right| \\
&=   \sum_{k \in \mathbb{Z}} \widehat{f}(k)^{1/2} \widehat{f}(k)^{1/2}   \left| \sum_{\ell=1}^{n}{e^{-2 \pi i k x_{\ell}}} \right|  \\
&\le \left( \sum_{k \in \mathbb{Z}} \widehat{f}(k) \right)^{1/2}  \left( \sum_{k \in \mathbb{Z}}{ \widehat{f}(k)\left|\sum_{m = 1}^{n}{e^{2\pi i k x_m}}\right|^2  } \right)^{1/2} \\
&=\sqrt{f(0)} \left( \sum_{m, \ell= 1}^{n}{f(x_m - x_{\ell}) } \right)^{1/2}.
\end{align*}

 Coupled with the inequality ($\diamond$) above, we obtain
 $$\left\| \sum_{\ell=1}^{n}{f(x-x_{\ell})} \right\|_{L^{\infty}} \le f(0)\sqrt{n}\lesssim \sqrt n,$$
 which was the desired statement.  To prove Theorem 1 and Corollary 2, we may further assume $\widehat f(k)\ge c|k|^{-2}$ for all $k\ne 0$, and thus
\begin{align*}
 n^2   \sum_{k \in \mathbb{Z} \atop k \neq 0}{\frac{c}{k^2} \left|\frac{1}{n}\sum_{m = 1}^{n}e^{2\pi i kx_m}\right|^2} &
 \le \sum_{k \in \mathbb{Z}}{\widehat f(k) \left|\sum_{m = 1}^{n}e^{2\pi i kx_m}\right|^2}\\
 &=\sum_{m,\ell=1}^nf(x_m-x_\ell)\le nf(0),
 \end{align*}
  so we have
 
 $$ \sum_{k \in \mathbb{Z} \atop k \neq 0}{\frac{c}{k^2} \left|\frac{1}{n}\sum_{m = 1}^{n}e^{2\pi i k x_m}\right|^2  } \leq \frac{f(0)}{n}.$$
Reformulating, 
$$   \left( \sum_{k \in \mathbb{Z} \atop k \neq 0}{\frac{1}{k^2} \left|\frac{1}{n}\sum_{m = 1}^{n}{e^{2\pi i k x_m}}\right|^2  } \right)^{1/2} \lesssim_{c, f(0)} \frac{1}{\sqrt{n}}.$$
 We note that this last argument has been previously stated in the literature in a very different context (integration error of periodic functions in terms of Zinterhof's diaphony \cite{zinterhof}) in a paper of Zinterhof \& Stegbuchner \cite{zinterhof2}. It remains to prove the second Corollary and Theorem 1.
For that we use an estimate of Peyr\'e \cite{peyre} (the one-dimensional case of Peyr\'e's inequality is an identity and can also be found in \cite[Exercise 64]{santa}): this estimate states that,
for any measure $\mu$ on $\mathbb{T}$
$$ W_2(\mu, dx) \lesssim \|\mu\|_{\dot H^{-1}} = \left( \sum_{k=1}^{\infty}{ \frac{|\widehat{\mu}(k)|^2}{k^2} }\right)^{1/2}.$$
We apply this estimate to the measure
$$ \mu = \frac{1}{n} \sum_{k=1}^{n}{\delta_{x_k}}$$
to obtain
$$ W_2\left( \frac{1}{n} \sum_{k=1}^{n}{\delta_{x_k}}, dx\right) \lesssim\left( \sum_{k=1}^{\infty}{ \frac{1}{k^2} \left| \frac{1}{n}\sum_{\ell=1}^{n}{e^{2 \pi i k x_{\ell}}} \right|^2 }\right)^{1/2} \lesssim \frac{1}{\sqrt{n}}.$$
This establishes Theorem 1. Corollary 2 follows from remarking that, as seen above,
$$
 \frac{1}{n^2}\sum_{m, \ell= 1}^{n}{f(x_m - x_{\ell}) } =\sum_{k \in \mathbb{Z}}{\widehat{f}(k) \left|\frac{1}{n}\sum_{m = 1}^{n}e^{2\pi i k x_m}\right|^2 }.$$
Moreover, since $\widehat{f}(k) \geq c|k|^{-2}$ for $k \neq 0$ and $\widehat{f}(0)=0$, we can bound this quantity from below by
$$ \sum_{k \in \mathbb{Z}}{\widehat{f}(k) \left|\frac{1}{n}\sum_{m = 1}^{n}e^{2\pi i k x_m}\right|^2 } \gtrsim
\sum_{k \neq 0}{\frac{1}{|k|^2} \left|\frac{1}{n}\sum_{m = 1}^{n}e^{2\pi i k x_m}\right|^2 }.$$
This, in turn, is the $\dot H^{-1}-$norm of $\mu$ and therefore, from another application of Peyr\'{e}'s inequality,
$$ \sum_{k \neq 0}{\frac{1}{|k|^2} \left|\frac{1}{n}\sum_{m = 1}^{n}e^{2\pi i k x_m}\right|^2 } \gtrsim W_2(\mu, dx).$$
\end{proof}

\section{Proof of Theorem 2}
\begin{lemma} We have
$$ \left\| \sum_{k=1}^{n} f(x-x_k) \right\|_{L^1} \leq 2f(0) \qquad \emph{for infinitely many}~n.$$
\end{lemma}
\begin{proof}Let us suppose that the inequality fails for some fixed $n$. Then, since
$$\int_{ \mathbb{T}}  \sum_{k=1}^{n} f(x-x_k) dx = 0,$$
we have that the positive mass and the negative mass cancel and thus, by pigeonholing,
$$ \min_{x \in \mathbb{T}}  \sum_{k=1}^{n} f(x-x_k) \leq -\frac{1}{2}\left\| \sum_{k=1}^{n} f(x-x_k)\right\|_{L^1} \leq -f(0).$$
This, in turn, then implies that
\begin{align*}
\sum_{k, \ell=1}^{n+1} f(x_k - x_{\ell}) &=  \sum_{k, \ell=1}^{n} f(x_k - x_{\ell}) + f(0) + 2 \sum_{k=1}^{n}{f(x_{n+1} - x_k)}\\
&\leq   \sum_{k, \ell=1}^{n} f(x_k - x_{\ell}) - f(0)
\end{align*}
and we see that the quantity is decaying since $f(0) > 0$.
However, the quantity cannot decay indefinitely since
$$ \sum_{k, \ell=1}^{n} f(x_k - x_{\ell})  = \sum_{k \in \mathbb{Z} \atop k \neq 0}{\widehat{f}(k) \left|\sum_{m = 1}^{n}e^{2\pi i k x_m}\right|^2  } \geq 0.$$
This means that the desired inequality has to eventually be true. The argument shows slightly more: since
$$   \sum_{k, \ell=1}^{n} f(x_k - x_{\ell})  \leq f(0)n,$$
we can infer that if the $L^1-$norm is bigger than $2f(0)$ for some fixed $n$, then it holds true for some $m \leq 2n$. However, we will not need this refined information.
\end{proof}

\begin{proof}[Proof of Theorem 2] 

We now fix such a value of $n$ where the $L^1-$norm is smaller than $2f(0)$. We argue that
\begin{align*}
 \left\| \frac{d}{dx} \sum_{m=1}^{n} f(x-x_m) \right\|_{L^2} &=  \left\| \frac{d}{dx} \sum_{k \in \mathbb{Z}}^{} \widehat{f}(k) \sum_{m = 1}^{n}e^{2\pi i k (x-x_m)} \right\|_{L^2} \\
 &\lesssim  \left\| \frac{d}{dx} \sum_{k \in \mathbb{Z}}^{} \frac{1}{k^2} \sum_{m = 1}^{n}e^{2\pi i k (x-x_m)} \right\|_{L^2}\\
&\lesssim  \left\|  \sum_{k \in \mathbb{Z}}^{} \frac{1}{k} \sum_{m = 1}^{n}e^{2\pi i k (x-x_m)} \right\|_{L^2}\\
&= \left( \sum_{k \neq 0} \frac{1}{k^2} \left| \sum_{m = 1}^{n}e^{2\pi i k x_m} \right|^2 \right)^{1/2} \\
&\lesssim \left(  \sum_{k, \ell=1}^{n} f(x_k - x_{\ell})\right)^{1/2}  \lesssim n^{1/2}.
\end{align*}
The final ingredient in our argument is a Gagliardo-Nirenberg inequality: for differentiable $g:\mathbb{T} \rightarrow \mathbb{R}$ with mean value 0, we have
$$ \| g\|_{L^{\infty}(\mathbb{T})} \lesssim \left\| \frac{d}{dx} g\right\|_{L^2(\mathbb{T})}^{2/3} \|g\|_{L^1(\mathbb{T})}^{1/3},$$
which establishes the desired result.
\end{proof}

\subsection{Proof of the Proposition}
\begin{proof} We have
 $$\sum_{m, \ell= 1}^{n}{f(x_m - x_{\ell}) } = \sum_{k \in \mathbb{Z} \atop k \neq 0}{\widehat{f}(k) \left|\sum_{m = 1}^{n}e^{2\pi i k x_m}\right|^2  } \lesssim  \sum_{k =1}^{\infty}{ \frac{1}{k^2} \left|\sum_{m = 1}^{n}e^{2\pi i k x_m}\right|^2  }.$$
This quantity was already estimated in \cite{stein0}, we recall the argument for the convenience of the reader. We observe that, trivially,
$$ \sum_{k =n^2}^{\infty}{  \frac{1}{k^2} \left|\sum_{m = 1}^{n}e^{2\pi i k x_m}\right|^2  } \leq n^2 \sum_{k=n^2}^{\infty}\frac{1}{k^2} \lesssim 1.$$
It thus remains to estimate the first $n^2$ sums. We split these sums into dyadic pieces and estimate
$$  \sum_{2^{\ell} \leq k \leq 2^{\ell+1}}^{}{ \frac{1}{k^2} \left|\sum_{m = 1}^{n}e^{2\pi i k x_m}\right|^2  } \lesssim \frac{1}{2^{2\ell}}   \sum_{2^{\ell} \leq k \leq 2^{\ell+1}}^{}{  \left|\sum_{m = 1}^{n}e^{2\pi i k x_m}\right|^2  }.$$
We recall the geometric series and use it to estimate
$$ \left|  \sum_{m=1}^{n}{e^{2 \pi i k m \alpha}} \right| = \left| \frac{e^{2\pi i k n \alpha } - 1}{e^{2\pi i k \alpha} - 1} \right| \leq \frac{2}{\left| e^{2\pi i k \alpha} - 1 \right|} \lesssim \frac{1}{ \left\| k \alpha \right\|},$$ 
where $\| x \| = \min(x-\left\lfloor x \right\rfloor, \left\lceil x \right\rceil - x)$ is the distance to the nearest integer. Since $\alpha$ is badly approximable, i.e.
$$ \left| \alpha - \frac{p}{q} \right| \geq \frac{c_{\alpha}}{q^2},$$
we have that, for any $2^{\ell} \leq k_1< k_2 \leq 2^{\ell+1}$, 
$$ \left| \|k_1 \alpha\| - \|k_2 \alpha\| \right| \geq \frac{c_\alpha}{2^{\ell+1}}.$$
Moreover, we also have
$$ \frac{c_\alpha}{2^{\ell+1}} \leq  \|k_1 \alpha\| ,  \|k_2 \alpha\|  \leq 1-  \frac{c_\alpha}{2^{\ell+1}}.$$
This shows that the sum
 $$ \sum_{2^{\ell} \leq k \leq 2^{\ell+1}}^{}{  \left|\sum_{m = 1}^{n}e^{2\pi i k x_m}\right|^2  } \lesssim   \sum_{2^{\ell} \leq k \leq 2^{\ell+1}}^{}{  \frac{1}{\|k\alpha\|^2}}$$
 can be estimated from above by
 $$ \sum_{2^{\ell} \leq k \leq 2^{\ell+1}}^{}{  \frac{1}{\|k\alpha\|^2}} \lesssim_{\alpha} \sum_{k=1}^{2^{\ell}}{ \frac{1}{ (k/2^{\ell})^2} }\lesssim 2^{2\ell}.$$
 Altogether this shows that over every dyadic block
 $$  \sum_{2^{\ell} \leq k \leq 2^{\ell+1}}^{}{ \frac{1}{k^2} \left|\sum_{m = 1}^{n}e^{2\pi i k x_m}\right|^2  } \lesssim \frac{1}{2^{2\ell}} \lesssim 1$$
 and thus the sum simplifies to the number of dyadic blocks up to $n^2$ which is $\sim \log{n}$.
\end{proof}

\subsection{Proof of Theorem 3}
\begin{proof} We can see that $K$ is positive-definite and equivalent to the Green's function $G$. Thus, it suffices to prove the desired result for the Green's function $G$ instead. The proof follows by induction from the main result of \cite{stein3}. Fixing a $d-$dimensional manifold $(M,g)$ with $d\geq 3$, we have for any set of $n$ points
$\left\{x_1, \dots, x_n\right\} \subset M$ that
$$ W_2\left( \frac{1}{n} \sum_{k=1}^{n}{\delta_{x_k}}, dx\right) \lesssim_M  \frac{1}{n^{1/d}} +  \frac{1}{n} \left| \sum_{k \neq \ell} G(x_k, x_{\ell})\right|^{1/2}.$$
If the manifold is two-dimensional, $d=2$, then we have
$$ W_2\left( \frac{1}{n} \sum_{k=1}^{n}{\delta_{x_k}}, dx\right) \lesssim_M  \frac{\sqrt{\log{n}}}{n^{1/2}} +  \frac{1}{n} \left| \sum_{k \neq \ell} G(x_k, x_{\ell})\right|^{1/2}.$$
We emphasize that $G(\cdot, y)$ has mean value 0 and thus, by the usual argument, we obtain that
$$ \min_{x \in M} \sum_{k=1}^{n-1}{G(x,x_k)} \leq 0$$
and thus
$$ \sum_{k \neq \ell} G(x_k, x_{\ell}) \leq 0.$$
We recall the Corollary from \cite{stein3} which implies that for \textit{any} set of points
$$\sum_{k, \ell =1 \atop k \neq \ell}^{n} G(x_k, x_{\ell}) \gtrsim_M - n^{2-2/d}$$
for $d \geq 3$. 
If the manifold is two-dimensional, then we have the estimate
$$\sum_{k, \ell =1 \atop k \neq \ell}^{n} G(x_k, x_{\ell}) \gtrsim_M - n^{} \log{n}.$$
These two results combined imply the desired statement.
\end{proof}

\textbf{Acknowedgments.} Part of this work was carried out while the second author was attending Dagstuhl Seminar 12391 (`Algorithms and Complexity for Continuous Problems'), he is grateful to Vladimir Temlyakov for valuable discussions. The authors are also grateful to Igor Shparlinski for helpful comments and for suggesting one of the Corollary 1 as well as its proof.

\end{document}